\documentclass[11pt]{amsart}
\usepackage{lmodern}
\usepackage{amsmath, amsthm, amssymb, amsfonts}
\usepackage[normalem]{ulem}
\usepackage{hyperref}
\usepackage{mathtools}

\hypersetup{
    colorlinks=true,
    linkcolor=black,
    filecolor=magenta,      
    urlcolor=cyan,
    citecolor=magenta,
    pdftitle={Overleaf Example},
    pdfpagemode=FullScreen,
    }
    
\usepackage{mathrsfs}

\usepackage{verbatim} 
\usepackage{longtable}

\usepackage{mathtools}

\usepackage{tikz}
\usetikzlibrary{decorations.pathmorphing}
\tikzset{snake it/.style={decorate, decoration=snake}}

\usepackage{caption}

\usepackage{tikz-cd}
\usetikzlibrary{arrows}

\theoremstyle{plain}
\newtheorem{thm}{Theorem}[section]
\newtheorem{cor}[thm]{Corollary}
\newtheorem{lem}[thm]{Lemma}
\newtheorem{prop}[thm]{Proposition}

\newtheorem{question}[thm]{Question}

\theoremstyle{definition}

\theoremstyle{remark}
\newtheorem{rmk}[thm]{Remark}

\newcommand{\BC}{{\mathbb{C}}}

\newcommand{\BF}{{\mathbb{F}}}

\newcommand{\BP}{{\mathbb{P}}}

\newcommand{\CE}{{\mathcal E}}
\newcommand{\CF}{{\mathcal F}}

\newcommand{\CO}{{\mathcal O}}

\newcommand{\Sym}{{\textnormal{Sym}}}

\newcommand{\ch}{{\mathrm{ch}}}

\DeclareFontFamily{OT1}{rsfs}{}
\DeclareFontShape{OT1}{rsfs}{n}{it}{<-> rsfs10}{}
\DeclareMathAlphabet{\curly}{OT1}{rsfs}{n}{it}

\newcommand\Hom{\operatorname{Hom}}

\usepackage{tikz}
\usepackage{lmodern}
\usetikzlibrary{decorations.pathmorphing}

\begin{document}
\title[Cohomological $\chi$-dependence of ring structure]{Cohomological $\chi$-dependence of ring structure for the moduli of one-dimensional sheaves on $\mathbb{P}^2$}
\date{\today}

\newcommand\blfootnote[1]{%
  \begingroup
  \renewcommand\thefootnote{}\footnote{#1}%
  \addtocounter{footnote}{-1}%
  \endgroup
}

\author[W. Lim]{Woonam Lim}
\address{Utrecht University, Department of Mathematics}
\email{w.lim@uu.nl}

\author[M. Moreira]{Miguel Moreira}
\address{Massachusetts Institute of Technology, Department of Mathematics}
\email{miguel@mit.edu}

\author[W. Pi]{Weite Pi}
\address{Yale University, Department of Mathematics}
\email{weite.pi@yale.edu}


\begin{abstract}
     We prove that the cohomology rings of the moduli space $M_{d,\chi}$ of one-dimensional sheaves on the projective plane are not isomorphic  for general different choices of the Euler characteristics. This stands in contrast to the $\chi$-independence of the Betti numbers of these moduli spaces. As a corollary, we deduce that $M_{d,\chi}$ are topologically different unless they are related by obvious symmetries, strengthening a previous result of Woolf distinguishing them as algebraic varieties.

\end{abstract}

\maketitle

\setcounter{tocdepth}{1} 

\tableofcontents
\setcounter{section}{0}

\section{Introduction}



We work over the complex numbers $\mathbb{C}$. 

\subsection{Moduli of sheaves.}
\label{section0.1}
Fix two integers $d$ and $\chi$ with $d\geq 1$. We consider the moduli space $M_{d,\chi}$ of semistable one-dimensional sheaves $\mathcal{F}$ on $\mathbb{P}^2$ with
\[
[\mathrm{supp}(\mathcal{F})]= d\cdot H \in H_2(\mathbb{P}^2, \mathbb{Z}), \quad \chi(\mathcal{F}) = \chi.
\]
Here $H$ is the class of a line, \,$\mathrm{supp}(\mathcal{F})$ denotes the Fitting support, and the stability condition is with respect to the slope
\[
\mu(\mathcal{E}) = \frac{\chi(\mathcal{E})}{c_1(\mathcal{E})\cdot H} \in \mathbb{Q}.
\]
\quad Le Potier \cite{LeP} first studied this moduli space and showed that $M_{d,\chi}$ is an irreducible projective variety of dimension $d^2+1$, nonsingular at all stable points. In particular, when $d$ and $\chi$ are coprime, semistability coincides with stability and $M_{d,\chi}$ is nonsingular. Geometry and topology of the moduli space $M_{d,\chi}$ have been intensively studied from various perspectives; see \cite[Introduction]{PS} for a brief overview. In this paper, we are interested in the cohomology rings $H^*(M_{d,\chi}, \mathbb{C})$ for a fixed $d$ and \textit{different} Euler characteristics $\chi$ coprime to $d$. Under the coprime assumption, the cycle class map
\[
\mathrm{cl}: \mathrm{CH}^*(M_{d,\chi}) \longrightarrow H^{2*}(M_{d,\chi}, \mathbb{Z})
\]
is an isomorphism by \cite[Theorem 2]{Integral}. Hence in this case we use $A^*(-)$ to denote the even cohomology $H^{2*}(-, \mathbb{C})$, or equivalently the Chow ring $\mathrm{CH}^*(-, \mathbb{C})$ with complex coefficients.

\smallskip

The moduli spaces $M_{d,\chi}$ admit two types of symmetries:

\begin{enumerate}
    \item[(a)] The first type of symmetry is given by the isomorphism
    \[
    \psi_1 : M_{d,\chi} \xrightarrow{\sim} M_{d,\chi+d}, \;\; \mathcal{F} \mapsto \mathcal{F} \otimes \mathcal{O}_{\mathbb{P}^2}(1).
    \]

    \item[(b)] The second type of symmetry \cite[Theorem 13]{Mai10} is given by the duality isomorphism
    \[
    \psi_2 : M_{d,\chi} \xrightarrow{\sim} M_{d,-\chi}, \;\; \mathcal{F} \mapsto \CE \kern -1.5 pt \mathit{xt}^1(\CF, \omega_{\BP^2}).
    \]
\end{enumerate}

Thus given two Euler characteristics $\chi, \chi'$ satisfying $\chi \equiv \pm \chi' \mod d$, there is a natural isomorphism $M_{d,\chi} \simeq M_{d, \chi'}$. The following theorem states that this is essentially the only case.

\begin{thm}[{\cite[Theorem 8.1]{Woolf}}]
\label{woolf}
For $d\geq 3$, there is an isomorphism of algebraic varieties
\[
M_{d,\chi} \simeq M_{d,\chi'}
\]
if and only if $\chi \equiv \pm \chi' \mod d$.
\end{thm}

The main result of this paper is a cohomological version of Theorem \ref{woolf}:

\begin{thm}\label{main}
    For $d\geq 1$ and $\chi, \chi'$ coprime to $d$, there is an isomorphism of graded $\mathbb{C}$-algebras
    \[
    A^*(M_{d,\chi}) \simeq A^*(M_{d,\chi'})
    \]
    if and only if $\chi \equiv \pm \chi' \mod d$.
\end{thm}

We find this cohomological $\chi$-dependence of ring structure interesting, since the (intersection) Betti numbers of $M_{d,\chi}$ are \textit{$\chi$-independent}: in particular, we have an isomorphism
\[
A^*(M_{d,\chi})\simeq A^*(M_{d,\chi'})
\]
of \textit{graded vector spaces} for arbitrary $\chi$ and $\chi'$ coprime to $d$; see Theorem \ref{MS} for the precise statement. As a corollary of Theorem \ref{main}, we deduce that $M_{d,\chi}$ are in general topologically different, strengthening Theorem~\ref{woolf}:

\begin{cor}
 For any $\chi, \chi'$ coprime to $d$, the moduli spaces $M_{d,\chi}$ and $M_{d,\chi'}$ are not homeomorphic unless $\chi \equiv \pm \chi' \mod d$.
\end{cor}

\subsection{BPS invariants and $\chi$-independence.} One motivation to study the moduli spaces $M_{d,\chi}$ comes from enumerative geometry. Let $X= \mathrm{Tot}(K_{\mathbb{P}^2})$ be the local Calabi--Yau 3-fold given by the total space of the canonical bundle on $\mathbb{P}^2$. Considerations from physics \cite{GV} predict an action of the Lie algebra $\mathfrak{sl}_2\times \mathfrak{sl}_2$ on the cohomology of a certain moduli space of $D$-branes supported on degree $d$ curves in $X$, yielding double-indexed integral invariants 
\begin{equation}\label{BPS}
  n_{d}^{i,j}\in \mathbb{Z}  
\end{equation}
as the dimensions of the weight spaces of this $\mathfrak{sl}_2\times \mathfrak{sl}_2$-action. These are known as \textit{refined BPS invariants} of $X$, which are expected to refine curve counting invariants for $X$ defined via Gromov--Witten/Donaldson--Thomas/Pandharipande--Thomas theory \cite{PT}.

\smallskip


One proposal by Hosono--Saito--Takahashi \cite{HST}, Kiem--Li \cite{KL}, and Maulik--Toda \cite{MT} suggests a mathematical definition of the invariants (\ref{BPS}) by the \textit{perverse filtration} on the cohomology of moduli spaces of one-dimensional sheaves. More precisely, the moduli space $M_{d,\chi}$ admits a Hilbert--Chow morphism
\[
h: M_{d,\chi} \longrightarrow \mathbb{P}H^0(\mathbb{P}^2, \mathcal{O}_{\mathbb{P}^2}(d)),
\]
sending a sheaf to its Fitting support. This map is proper and induces an increasing filtration on the intersection cohomology
\[
P_0\, \mathrm{IH}^*(M_{d,\chi}) \subset P_1\, \mathrm{IH}^*(M_{d,\chi}) \subset \cdots \subset \mathrm{IH}^*(M_{d,\chi})
\]
called the perverse filtration; see \cite[Section 1.1]{P=C}. The invariants (\ref{BPS}) are defined in \textit{loc. cit.} as the dimension of the graded pieces of this filtration:
\begin{equation}
\label{defBPS}
n_{d}^{i,j}\coloneqq \dim \mathrm{Gr}_i^P \,\mathrm{IH}^{i+j}(M_{d,\chi}).
\end{equation}

For this to be well-defined, the RHS in (\ref{defBPS}) should not depend on the choice of $\chi$, which is a priori non-trivial in light of Theorem \ref{woolf}. This is a special case of Toda's cohomological $\chi$-independence conjecture \cite[Conjecture 1.2]{Toda}; see also \cite[Conjecture 0.4.3]{PB} for a version on the intersection Betti numbers of $M_{d,\chi}$. The following theorem of Maulik and Shen confirms this conjecture\footnote{The case where $\chi, \chi'$ are coprime to $d$ was previously proven by Bousseau \cite{PB}.}.

\begin{thm}[{\cite[Theorem 0.1]{MS_GT}}]\label{MS}
    For any $\chi, \chi' \in \mathbb{Z}$ not necessarily coprime to $d$, there is a (non-canonical) isomorphism of graded vector spaces
    \[
    \mathrm{IH}^*(M_{d,\chi}) \simeq \mathrm{IH}^*(M_{d,\chi'})
    \]
    preserving the perverse filtration on both sides.
\end{thm}

If we restrict to the coprime case, intersection cohomology coincides with singular cohomology which admits a canonical $\mathbb{Q}$-algebra structure. It is then natural to ask if we can choose the isomorphism of Theorem \ref{MS} to be an isomorphism of $\mathbb{Q}$-algebras. More generally, we can ask the following:

\begin{question}[{\cite[Section 0.3]{PS}}]\label{0.5}
    For any $\chi, \chi'$ coprime to $d$, is there an isomorphism
    \[
    H^*(M_{d,\chi},\mathbb{Q}) \simeq H^*(M_{d,\chi'}, \mathbb{Q})
    \]
    of graded $\mathbb{Q}$-algebras?
\end{question}

Theorem \ref{main} gives a complete answer to this question, in the stronger sense that we work with $\mathbb{C}$-coefficients instead of $\mathbb{Q}$.


\subsection{Structure of the proof of Theorem \ref{main}.} We outline the proof of our main result here, as the argument is of an elementary nature but somewhat long. It was explained in \cite{PS} that the rings $A^\ast(M_{d,\chi})$ admit a minimal set of $3d-7$ generators 
\[T=\{c_0(2), c_2(0)\}\cup \bigcup_{k=2}^{d-2}\{c_{k-1}(2), c_k(1), c_{k+1}(0)\}\,.\]
In other words, there is a surjective graded algebra homomorphism $\BC[T]\twoheadrightarrow  A^\ast(M_{d,\chi})$. This homomorphism is an isomorphism up to degree $d-1$, but in degree $d$ there are exactly 3 relations, cf. Theorem \ref{generation}. As a first step toward the proof, we show that these 3 relations $R_1, R_2, R_3$ spanning $\ker(\BC[T]^d\to A^d(M_{d,\chi}))$ can be obtained explicitly from the tautological relations of \cite[Section 2]{PS}, cf. Proposition \ref{prop: three relations are tautological.}.

 We then proceed by showing that if $\chi$ and $\chi'$ are not related by the symmetries (a) and (b), i.e., $\chi \not\equiv \pm \chi' \mod d$, then there is no automorphism of the graded algebra $\BC[T]$ that sends the 3-dimensional subspace of relations $\ker(\BC[T]^d\to A^d(M_{d,\chi}))\subseteq \BC[T]^d$ to $\ker(\BC[T]^d\to A^d(M_{d,\chi'}))$. To achieve this, we do not use the full relations $R_1, R_2, R_3$, whose sizes grow considerably as $d$ increases, but only a truncated version of those. Concretely, we project the subspaces of relations onto the 18-dimensional subspace 
\[c_{d-3}(2)\cdot \BC[T]^2\oplus c_{d-2}(1)\cdot \BC[T]^2\oplus c_{d-1}(0)\cdot\BC[T]^2\subseteq \BC[T]^d\,.\]
This truncated form of the relations is completely explicit: for each of the 3 relations, we only need to compute the 18 coefficients that are rational functions in $d$ and $\chi$. 

Showing that there is no automorphism of $\BC[T]$ relating the (truncated) subspaces of relations for $\chi$ and $\chi'$ then becomes a very concrete, but still involved, linear algebra problem. For this, we employ a determinant trick in \textit{Step 2} of Section \ref{sec: proof} that imposes strong restrictions on the possible automorphisms. The latter fall into two types, which we call \textit{Type I} and \textit{Type II} solutions.  We treat them separately in Sections \ref{sec: type1}--\ref{sec: type1extended}, showing that both cannot exist by analyzing certain numerical constraints (see for example Equation \eqref{constraint}) on the triple $(d,\chi, \chi')$.

\smallskip

Our proof of Theorem \ref{main} involves heavy computations aided by the software Mathematica~\cite{mathematica}. The code (together with a printed file) has been uploaded to the third author's website
\centerline{\url{https://github.com/Weite-Pi/weitepi.github.io}}
under the name
\texttt{cohomological ring chi-dependence}, which we shall frequently refer to in the later part of this paper. 

\subsection{Relations to other works.} The moduli space $M_{d,\chi}$ shares similar features with two types of other moduli spaces, the moduli of Higgs bundles\footnote{Recall that Higgs bundles on a curve $C$ are in correspondence with one-dimensional sheaves on the surface $\textup{Tot}(K_C)$ via the spectral correspondence \cite{BNR}. This analogy provides another motivation to study $M_{d,\chi}$; see \cite[Introduction]{PS} for details.} and the moduli of one-dimensional sheaves on K3 surfaces. The parallel of Question \ref{0.5} for those cases has a positive answer. Cohomology rings of the moduli of Higgs bundles are proven to be $\chi$-independent in the stronger sense that perverse filtrations are preserved, as predicted by the $P=W$ conjecture \cite{dCHM1, MS_PW, HMMS}. A direct proof using techniques in characteristic $p$ is given in \cite{dCMSZ}. The case of K3 surfaces follows from the fact that moduli of sheaves on K3 surfaces with respect to a primitive class and generic stability is birational to the Hilbert scheme of points \cite{BM} and that any two birational projective hyperkähler manifolds are deformation equivalent \cite{H}. 

Our main result asserts that the cohomology rings $H^{*}(M_{d,\chi}, \mathbb{C})$ are $\chi$-dependent in general; on the other hand, some \textit{$\chi$-independent} multiplicative structures have been observed or speculated: 

\begin{itemize}
    \item It is conjectured and proven under certain assumptions in \cite{BLM} that the \textit{Virasoro constraints} hold for the moduli spaces $M_{d,\chi}$. Briefly speaking, the Virasoro constraints predict that certain intersection numbers on $M_{d,\chi}$ obtained by integrating natural cohomology classes satisfy some explicit universal relations. See \cite[Section 1.3]{BLM} for the precise statement.

    \item The main theorem of \cite{PS} is a uniform\footnote{In the sense that the generators $c_k(j)$ in Theorem \ref{generation} are defined without explicit reference to $\chi$.} minimal generation and freeness result on $H^*(M_{d,\chi}, \mathbb{C})$ for all $\chi$ coprime to $d$; see Theorem \ref{generation} for a precise and slightly stronger statement. The $P=C$ conjecture formulated in \cite{P=C} seeks to identify the perverse filtration on the free part of $H^*(M_{d,\chi}, \mathbb{Q})$ with an explicit \textit{Chern filtration} defined in terms of the generators. This prediction is also independent of $\chi$. 
    
    \item Finally, we remark that the perverse filtration on $H^*(M_{d,\chi}, \mathbb{C})$ comes from the ring structure. Indeed, denoting by $L$ the pull-back of the hyperplane class by the Hilbert--Chow morphism, it is completely determined by the multiplication operator $L$ according to \cite[Proposition 5.2.4]{dCM0}.  

    
\end{itemize}

\subsection{Acknowledgments.} We thank Yakov Kononov for his help on coding and for numerous helpful conversations. The third author wishes to thank Rahul Pandharipande for inviting him to visit ETH Zürich, which eventually resulted in this collaboration, and thanks to Junliang Shen, his academic advisor, for proposing this interesting problem in the first place. 

The first author is supported by the grant SNF-200020-182181. The second author is supported by ERC-2017-AdG-786580-MACI. The project received funding from the European Research Council (ERC) under the European Union Horizon 2020 research and innovation programme (grant agreement 786580).

\section{Tautological classes and relations in \texorpdfstring{$A^d(M_{d,\chi})$}{AM}}

For the remainder of this paper, we assume $\chi, \chi'$ are coprime to $d$. This section provides some preliminary results on the cohomology of the moduli spaces $M_{d,\chi}$. We recall the normalized tautological classes and certain tautological relations introduced in \cite{PS}, and prove a first result on the relations in $A^d(M_{d,\chi})$.

\subsection{Tautological classes.}

Under the assumption that $\mathrm{gcd}(d,\chi)=1$, there exists a universal sheaf \cite[Theorem 4.6.5]{HL} over the product $\mathbb{P}^2 \times M_{d,\chi}$ which we denote by $\mathbb{F}$. For a stable sheaf $[\mathcal{F}]\in M_{d,\chi}$, the restriction of $\mathbb{F}$ to the fiber $\mathbb{P}^2 \times [\mathcal{F}]$ recovers $\mathcal{F}$.

Consider the two projection maps $\pi_P$ and $\pi_M$ from $\mathbb{P}^2 \times M_{d,\chi}$ to its first and second components. One way to obtain natural cohomology classes on $M_{d,\chi}$ is to take the Chern characters of $\mathbb{F}$, intersect with classes pulled back from $\mathbb{P}^2$, and push forward to $M_{d,\chi}$. The choice of $\mathbb{F}$ is however not unique: any two universal sheaves differ by tensoring with a line bundle pulled back from $M_{d,\chi}$. We thus conduct a normalization of $\mathbb{F}$ as follows.

\smallskip

For a universal sheaf $\mathbb{F}$ and a class
\[
\delta = \pi_P^* \delta_P + \pi_M^* \delta_M \in A^1(\mathbb{P}^2 \times M_{d,\chi}),\;\;\; \textrm{with \,} \delta_P \in A^1(\mathbb{P}^2), \;\; \delta_M \in A^1(M_{d,\chi}),
\]
we consider the \textit{twisted} Chern character
\[
\mathrm{ch}^\delta(\mathbb{F})\coloneqq \mathrm{ch}(\mathbb{F}) \cdot \exp(\delta),
\]
and denote by $\mathrm{ch}_k^\delta(\mathbb{F})$ its degree $k$-part. For any class $\gamma\in A^*(\mathbb{P}^2)$, we write
\[
\int_\gamma \mathrm{ch}_k^\delta(\mathbb{F})\coloneqq {\pi_M}_* (\pi_P^* \gamma \cdot \mathrm{ch}_k^\delta(\mathbb{F})) \in A^*(M_{d,\chi}). 
\]

 It is proven in \cite[Proposition 1.2]{PS} that for a fixed $\mathbb{F}$, there exists a {unique} class $\delta_0=\delta_0(\BF)$ as above satisfying the normalizing conditions
    \begin{equation}\label{norm}
    \int_H \mathrm{ch}_2^{\delta_0} (\mathbb{F})=0,\;\;\; \int_{\mathbf{1}_{\mathbb{P}^2}} \mathrm{ch}^{\delta_0}_2(\mathbb{F})=0.
    \end{equation}
    This class is computed explicitly in \cite[Proposition 2.1]{P=C}: 
    \begin{equation}
    \label{delta}
    \delta_0=\left(\frac{3}{2}-\frac{\chi}{d}\right)\cdot H-\frac{1}{d}\left(\left(\frac{3}{2}-\frac{\chi}{d}\right)\int_{H^2}\ch_1(\BF)+\int_H \mathrm{ch}_2(\mathbb{F})\right).
    \end{equation}

\smallskip

Note that the twisted Chern character $\ch^{\delta_0}(\BF)$ does \textit{not} depend on the choice of $\BF$ anymore. With this normalizing class, we define the \textit{tautological classes}
\begin{equation}
\label{tautclass}
c_k(j) \coloneqq \int_{H^j} \mathrm{ch}_{k+1}^{\delta_0}(\mathbb{F}) \in A^{k+j-1}(M_{d,\chi}).
\end{equation}

\smallskip

We collect some basic properties of the tautological classes from \cite[Section 1.2]{PS}.

\begin{prop}\label{prop2.2}
    Let $c_k(j)$ be the tautological classes defined by (\ref{tautclass}).
    \begin{enumerate}
        \item[(a)] The classes $c_k(j)$ do not depend on the choice of a universal sheaf.

        \item[(b)] We have
        \[
        c_1(0)=0 \in A^0(M_{d,\chi}), \;\;\; c_1(1) = 0 \in A^1(M_{d,\chi}),\;\;\; c_0(1)=d\in A^0(M_{d,\chi}).
        \]

        \item[(c)] Let $\psi_1$ and $\psi_2$ be the two symmetries introduced in Section \ref{section0.1}. Then we have
        \[
        \psi_1^* c_k(j)= c_k(j),\;\;\; \psi_2^* c_k(j) = (-1)^k c_k(j).
        \]
    \end{enumerate}
\end{prop}

\smallskip

The next theorem gives a first result on the structure of $A^*(M_{d,\chi})$ in terms of the tautological classes. Note that we change the coefficient from $\mathbb{Q}$ to $\mathbb{C}$. 

\begin{thm}[{\cite{PS, YY5}}]\label{generation} Assume $d \geq 5$. We have:
\begin{enumerate}
    \item[(a)] $A^*(M_{d,\chi})$ is generated as a $\BC$-algebra by the $3d-7$ classes\footnote{We call these $3d-7$ classes the \textit{tautological generators.}} of degrees $\leq d-2$:
    \begin{equation}\label{taut}
   c_0(2),\, c_2(0) \in A^1(M_{d,\chi}), \;\;\;  c_{k}(0),\, c_{k-1}(1),\, c_{k-2}(2) \in A^{k-1}(M_{d,\chi}),\;\; 3\leq k \leq d-1.
    \end{equation}
    \item[(b)] There is no relation among these $3d-7$ classes in degrees $\leq d-1$.
    \item[(c)] There are exactly three linearly independent relations in degree $d$.
\end{enumerate}

\end{thm}

We briefly recall the proof of Theorem \ref{generation}: (i) By the results of Beauville \cite{Beau} or Markman \cite{Integral}, the ring $A^*(M_{d,\chi})$ is generated by the tautological classes. (ii) Using the geometry of $M_{d,\chi}$, we produce tautological relations that express any tautological class in terms of the $3d-7$ classes in (\ref{taut}); see the paragraph after Proposition \ref{prop2.8}. (iii) Part (b) and (c) follow from comparing the Betti numbers of $M_{d,\chi}$ with those of Hilbert schemes of points on $\mathbb{P}^2$; this is explained in Yuan \cite{YY4, YY5}.

\medskip



As we will see, the three relations in $A^d(M_{d,\chi})$ are the central characters in our proof of Theorem \ref{main}. By carefully investigating the $\chi$-dependence of the three relations, we are able to deduce that $A^*({M_{d,\chi}})$ are non-isomorphic for different choices of $\chi$ unless they are related by the symmetries in Section \ref{section0.1}. 

\begin{rmk}
The choice of the normalization (\ref{norm}) and the shift in degrees of the Chern character in (\ref{tautclass}) are motivated by the $P=C$ conjecture \cite{P=C}. Roughly speaking, the $P=C$ conjecture predicts that the \textit{Chern grading} of the tautological generators $c_k(j)$ coincides with its \textit{perversity}; this gives a conjectural explicit description of the perverse filtration on $A^{*\leq d-2}(M_{d,\chi})$ in terms of (\ref{taut}). See \cite[Conjecture 0.3 and Proposition 1.2]{P=C} for details.
\end{rmk}

\subsection{Tautological relations}

We recall in this section certain tautological relations \cite[Section 1.2]{PS} on the cohomology of $M_{d,\chi}$. By the symmetry (a) in Section \ref{section0.1}, we may assume that $0 < \chi < d$, so that any $\mathcal{F} \in M_{d,\chi}$ satisfies $0< \mu(\mathcal{F}) <1$.

Consider the triple product $Y\coloneqq \mathbb{P}^2 \times M_{d,\chi} \times \check{\mathbb{P}}^2$, where $\check{\mathbb{P}}^2$ is the dual projective space parametrizing lines in $\mathbb{P}^2$. Let $\pi_R : Y \to \check{\mathbb{P}}^2$ be the third projection. We write $p= \pi_P \times \pi_M, \, q = \pi_P \times \pi_R$, and $r = \pi_M \times \pi_R$:
\[\begin{tikzcd}
	& Y \\
	{\BP^2\times M_{d,\chi}} & {M_{d,\chi} \times \check{\mathbb{P}}^2} & {\BP^2\times \check{\mathbb{P}}^2}
	\arrow["p"', from=1-2, to=2-1]
	\arrow["r"', from=1-2, to=2-2]
	\arrow["q", from=1-2, to=2-3].
\end{tikzcd}\]

\medskip

\noindent Let $Z\subset \mathbb{P}^2\times \check{\mathbb{P}}^2$ be the incidence subscheme, and let $\mathcal{O}_Z$ be its structure sheaf. 
For a fixed universal sheaf $\mathbb{F}$ on $M_{d,\chi}$, we consider the complex
\begin{equation}\label{H(n)}
    \mathcal{H}(n)\coloneqq R \mathcal{H}\kern -.8pt \mathit{om} ( p^*\mathbb{F}, q^*\mathcal{O}_Z\otimes \pi_P^*\mathcal{O}_{\mathbb{P}^2}(-n)) \in D^b \mathrm{Coh}(Y).
\end{equation}
 
The projection $r: Y \to M_{d,\chi} \times \check{\mathbb{P}}^2$ is a trivial $\mathbb{P}^2$-bundle, so the derived push-forward
\[
Rr_*\mathcal{H}(n)\in D^b(M_{d,\chi} \times \check{\mathbb{P}}^2)
\]
admits a three-term resolution $0\to K^0 \to K^1 \to K^2 \to 0$ by vector bundles.

\begin{lem}
For $n\in \{1,2,3\}$, we can choose $K^i$ with $K^0 = K^2 =0$ and $K^1$ free of rank $d$. 
\end{lem} 

\begin{proof}
    This is \cite[Lemma 2.4]{PS}; we briefly recall the proof here. Take a point 
 \[
 P=([\mathcal{F}],[L])\in M_{d,\chi}\times \check{\mathbb{P}}^2,
 \]
 where $[L]\in \check{\BP}^2$ corresponds to the line $L\subseteq \BP^2$. Over the point $P$, cohomology of the complex $K^0(P) \to K^1(P)\to K^2(P)$ computes the extension groups
 \[
 \mathrm{Ext}^i(\mathcal{F}, \mathcal{O}_{L}(-n)), \quad i=1,2,3.
 \] Note that $\mu(\mathcal{O}_{L}(-n))=1-n$. By stability and Serre duality, one checks that
 \[
 \mathrm{Ext}^0(\CF, \mathcal{O}_{L}(-n))=0, \quad \mathrm{Ext}^2(\mathcal{F}, \CO_{L}(-n))=0
 \] 
 as long as $n\in \{1,2,3\}$. This implies $Rr_*\mathcal{H}(n)$ can be represented by a single vector bundle $K^1$ concentrated in degree $1$, whose rank is $d$ by a Hirzebruch--Riemann--Roch calculation.
\end{proof}

It follows that $-Rr_* \mathcal{H}(n)$ is a rank $d$ vector bundle on $M_{d,\chi} \times \check{\mathbb{P}}^2$, whence
\begin{cor}
\label{cor2.7}
For $\ell \geq d+1$ and $n\in \{1,2,3\}$, we have
\begin{equation}\label{chernrel}
 c_\ell(-Rr_*\mathcal{H}(n))=0 \in A^{\ell}(M_{d,\chi} \times \check{\BP}^2).
\end{equation}
\end{cor}

\smallskip

Recall that $\mathcal{H}(n)$ is defined by (\ref{H(n)}) in terms of a fixed universal sheaf $\mathbb{F}$. If we replace $\mathbb{F}$ by
\[
\mathbb{F}'\coloneqq \mathbb{F} \otimes \pi_M^* L,
\]
with $L$ a \textit{fractional} line bundle over $M_{d,\chi}$, then according to (\ref{H(n)}) and the projection formula, the vector bundle $-Rr_*\mathcal{H}(n)$ is replaced by its tensor with the dual line bundle $\pi_M^* L^\vee$. Using a formal argument involving Chern roots, one checks that the identity (\ref{chernrel}) also holds for $\mathbb{F}'$. 

\smallskip

In particular, we choose the fractional line bundle $L$ to have first Chern class\footnote{One can check that $\int_{H^2}\ch_1(\BF)=c_0(2)$ for \textit{every} universal sheaf $\mathbb{F}$.}
\[
c_1(L)=-\frac{1}{d}\left(\left(\frac{3}{2}-\frac{\chi}{d}\right)c_0(2)+\int_H \mathrm{ch}_2(\mathbb{F})\right),
\]
so that the normalizing class (\ref{delta}) for $\mathbb{F}'$ is given by $\delta_0(\BF')=\left(\frac{3}{2}-\frac{\chi}{d}\right)\cdot H$. Now we apply (\ref{chernrel}) to $\mathbb{F}'$. The Chern classes in (\ref{chernrel}) can be expressed in terms of the Chern character $\mathrm{ch}(Rr_*\mathcal{H}(n))$, which is computed by the Grothendieck--Riemann--Roch theorem:
\begin{align*}
\mathrm{ch}(Rr_*\mathcal{H}(n))&=r_*(\mathrm{ch}(\mathcal{H}(n))\cdot \mathrm{td}(\BP^2)) \\
&=r_*(\mathrm{ch}(p^*{\mathbb{F}'}^\vee)\cdot \mathrm{ch}(q^*\mathcal{O}_Z\otimes \pi_P^* \mathcal{O}_{\BP^2}(-n))\cdot \mathrm{td}(\BP^2)).
\end{align*}

We denote by $\beta$ the class of a line on $\check{\mathbb{P}}^2$. Expanding the right-hand side of the above equation, Corollary \ref{cor2.7} then gives relations in $A^*(M_{d,\chi} \times \check{\mathbb{P}}^2)$:


\begin{prop}\label{prop2.8}
For every $\ell \geq d+1$ and $n\in \{1,2,3\}$, the following identity holds:
\begin{equation}
 \label{relation}
 \sum_\mathbf{m} \prod_{s=1}^\ell \frac{((s-1)!)^{m_s}}{(m_s)!}\left( \pi_M^* A_{s}- \sum_{0\leq i\leq 2}\frac{\pi_R^*\beta^i}{i!}(-1)^i \pi_M^* B_{s-i}\right)^{m_s}=0.
 \end{equation}
 Here, the first sum is over all $\ell$-tuple of non-negative integers $\mathbf{m}=(m_1,m_2,\ldots, m_\ell)$ such that $m_1+2m_2+\cdots +\ell m_\ell=\ell$, and writing $\widetilde{c}_s(j)\coloneqq (-1)^{s+1}c_s(j)$, the terms $A_s, B_s$ are given by
\small
\begin{align*}
    A_s &\coloneqq\widetilde{c}_{s+1}(0)+\left(3 - n - \frac{\chi}{d}\right)\widetilde{c}_{s}(1)+\frac{1}{2d^2}\left(\left(n - \frac{7}{2}\right)d + \chi\right)\left(\left(n - \frac{5}{2}\right)d + \chi\right)\widetilde{c}_{s-1}(2),\\
    B_s &\coloneqq\widetilde{c}_{s+1}(0)+\left(2-n-\frac{\chi}{d}\right)\widetilde{c}_{s}(1)+\frac{1}{2d^2}\left(\left(n - \frac{5}{2}\right)d + \chi\right)\left(\left(n - \frac{3}{2}\right)d + \chi\right)\widetilde{c}_{s-1}(2).
\end{align*}
\end{prop}

\begin{proof}
    The proof is essentially the same as \cite[Proposition 2.6]{PS}, except that we use the normalized Chern character $\mathrm{ch}^{\delta_0}({\mathbb{F}'}^\vee)=\mathrm{ch}({\mathbb{F}'}^\vee)\cdot \exp(-\delta_0)$. The convoluted formula results from applying Newton's identity to express Chern classes in the Chern characters, and the class $\beta$ comes from $\mathrm{ch}(q^* \mathcal{O}_Z)$. See also \cite[Proposition 2.7]{P=C} for an explicit version for $M_{4,1}$. 
\end{proof}

To obtain relations in $A^*(M_{d,\chi})$, we first take the identity (\ref{relation}) for some $\ell \geq d+1$; this gives a relation in $A^\ell(M_{d,\chi} \times \check{\mathbb{P}}^2)$. Then we multiply it by $\pi_R^* (\beta^j)$ and push forward to $M_{d,\chi}$, where $0\leq j \leq 2$. This produces a relation in $A^{\ell+j-2}(M_{d,\chi})$ among the tautological classes (\ref{tautclass}). The procedure is explained in detail in \cite[Section 2.3]{PS} and the paragraph before it. Indeed, as a key step in the proof of Theorem \ref{generation}(a), it is shown in \cite[Section 2.3]{PS} that the relations produced by Proposition \ref{prop2.8} express all the tautological classes (\ref{tautclass}) in terms of the $3d-7$ generators (\ref{taut}).

\subsection{Three relations in $A^{d}(M_{d,\chi})$.} \label{sec2.3}
The goal of this section is to prove the following result:

\begin{prop}\label{prop2.11}\label{prop: three relations are tautological.}
    For $d\geq 5$, the three linearly independent relations in $A^d(M_{d,\chi})$, cf. Theorem \ref{generation}(c), can be produced by Proposition \ref{prop2.8}.
\end{prop}

We first introduce some notations for the rest of the paper. Let $T^k$ be the linear space spanned by tautological classes in degree $k$, that is
\[T^1=\textup{span}_\BC\{c_2(0), c_0(2)\}\,,\quad T^k=\textup{span}_\BC\{c_{k+1}(0),c_{k}(1), c_{k-1}(2)\}\,.\]

Let 
\[
T\coloneqq\bigoplus_{k=1}^{d-2} T^k
\]
be the span of the $3d-7$ tautological generators in (\ref{taut}), and let
\[
T^+ \coloneqq T \oplus T^{d-1}\oplus T^d.
\]
We denote by $\mathbb{C}[T]$ the symmetric algebra in $T$ (i.e. the free algebra generated by the $3d-7$ tautological generators), and by $\mathbb{C}[T]^k$ its graded piece of degree $k$. Similarly we write $\BC[T^+]$ for the symmetric algebra on $T^+$.

\smallskip

Furthermore, we endow the basis of $T^+$ with a total ordering $\prec$ as follows: first order the basis by cohomological degrees; within each degree, we order by the Chern grading of the basis. In a precise way, this means $c_k(j) \prec c_{k'}(j')$ if and only if $k+j-1<k'+j'-1$, or $k+j-1=k'+j'-1$ and $k<k'$. Finally, we endow monomials in each degree of $\mathbb{C}[T^+]$ the lexicographical order induced by the total ordering $\prec$ on $T^+$. In this way, we can talk about the leading term of a polynomial in $\mathbb{C}[T^+]$, as well as in $\mathbb{C}[T]$. 

\medskip
 We have graded algebra homomorphisms
\[\BC[T]\hookrightarrow \BC[T^+]\to A^\ast(M_{d,\chi})\,.\]
Theorem \ref{generation} states that the homomorphism $\BC[T]\to A^\ast(M_{d,\chi})$ is surjective, that it is an isomorphism up to degree $d-1$, and that 
$\ker\!\big(\BC[T]^d\to A^d(M_{d,\chi})\big)$
is 3-dimensional. A relation is by definition an element in the kernel of $\BC[T]\to A^\ast(M_{d,\chi})$ or $\BC[T^+]\to A^\ast(M_{d,\chi})$.

\medskip

With the above setup, Theorem \ref{generation} (b) and (c) are equivalent to the following\footnote{See \cite[Theorem 1.1]{YY5}, where the relevant Betti numbers $b_{2k}({\mathbb{P}^2}^{[n]})$ of the Hilbert scheme equal $\dim \mathbb{C}[T^+]^k$ by Göttsche's formula \cite{Go1}.}.

\begin{prop}
    For $d\geq 5$, we have
    \[
    \dim A^k(M_{d,\chi}) = \begin{cases}
        \dim \mathbb{C}[T^+]^k, \quad & k \leq d-2,\\
        \dim \mathbb{C}[T^+]^k-3, \quad &k = d-1,\\
        \dim \mathbb{C}[T^+]^k -12, \;\;
        &k=d.
    \end{cases}
    \]
\end{prop}

Since $\dim \BC[T^+]^d=\dim \BC[T]^d+9$, to prove Proposition \ref{prop2.11} it suffices to produce 12 linearly independent relations in $\mathbb{C}[T^+]^d$ from Proposition \ref{prop2.8}. These 12 relations in degree $d$ are produced in the following three ways:
\begin{enumerate}
    \item[(a)] Take $\ell = d+1$ in \eqref{relation}, we get the vanishing of a cohomology class in $M_{d, \chi}\times \check{\BP}^2$. Pushing forward to $M_{d, \chi}$ yields a relation $R_{{(a)}}^n\in \BC[T^+]^{d-1}$ for $n=1,2,3$. Each of these 3 relations produces 2 relations in degree $d$, namely $c_2(0)R_{(a)}^n$ and $c_0(2)R_{(a)}^n$.
    \item[(b)] Take $\ell = d+1$ in (\ref{relation}), multiply by $\pi_R^* (\beta)$ and push forward to $M_{d, \chi}$. This produces 3 relations $R_{(b)}^n\in \BC[T^+]^{d}$ for $n=1,2,3$. 
    \item[(c)] Take $\ell = d+2$ in (\ref{relation}) and push forward to $M_{d, \chi}$. This produces 3 relations $R_{(c)}^n\in \BC[T^+]^{d}$ for $n=1,2,3$. 
\end{enumerate}
We now proceed to argue that these 12 relations are linearly independent. 
It can be checked, cf. \texttt{det1} in the Mathematica file and \cite[Section 2.3]{PS}, that under the assumption $\mathrm{gcd}(d,\chi)=1$, the $3\times 3$ matrix of coefficients of $c_{d}(0), c_{d-1}(1), c_{d-2}(2)$ in the three relations $R_{(a)}^1, R_{(a)}^2, R_{(a)}^3$ is \textit{nonsingular;} in fact, it has determinant
\[\texttt{det1}=
\frac{(-1)^{d}(d - 2)^4(d-1) }{4}\chi(d - \chi)(d-2\chi)\,.
\]
\smallskip
In particular, equations $R_{(a)}^n$ can be used to express $c_{d}(0), c_{d-1}(1), c_{d-2}(2)$ in terms of the tautological generators (\ref{taut}). Multiplying these equations with $c_2(0), c_0(2)$, we obtain 6 relations in $\mathbb{C}[T^+]^{d}$. These are still linearly independent since the $6\times 6$ matrix given by the coefficients of the monomials
\[\mathsf{Mon}_1=\{c_{d}(0)c_2(0), c_{d}(0)c_0(2), c_{d-1}(1)c_2(0), c_{d-1}(1)c_0(2), c_{d-2}(2)c_2(0), c_{d-2}(2)c_0(2)\}\,
\]
is invertible with determinant $\texttt{det1}^2$.

\smallskip

We consider now the 6 relations $R_{(b)}^n, R_{(c)}^n$, for $n=1,2,3$ and the set of monomials 
\[\mathsf{Mon}_2=\{c_{d+1}(0), c_{d}(1), c_{d-1}(2), c_{d-1}(0)c_3(0), c_{d-1}(0)c_2(1),c_{d-1}(0)c_1(2)\}\,.\]
The $6\times 6$ matrix of coefficients of $\mathsf{Mon}_2$ in these 6 relations is non-singular. Indeed, a direct computation (cf. \texttt{det2} in the Mathematica file) gives the determinant
\[
\texttt{det2}=4(d-2)^6(d-1)^3 d^4\neq 0\,.
\]
It follows that the 6 relations $R_{(b)}^n$ and $R_{(c)}^n$ with $n\in \{1,2,3\}$ are linearly independent.

\smallskip

Finally, we note that the monomials in $\mathsf{Mon}_2$ do not appear in the relations of the form $c_2(0)R_{(a)}^n, c_0(2)R_{(a)}^n$, so the $12\times 12$ matrix of coefficients of monomials $\mathsf{Mon}_1\cup \mathsf{Mon}_2$ in the 12 relations obtained by (a), (b), (c) is also nonsingular. We conclude that these 12 relations are all linearly independent, completing the proof of Proposition \ref{prop2.11}.

\medskip To explicitly obtain the 3 relations in $\BC[T]^d$, we proceed as follows. We use the relations $R_{(a)}^n$ to express $c_{d}(0), c_{d-1}(1), c_{d-2}(2)$ in terms of lower degree generators in the image of $\BC[T]$; we then use $R_{(b)}^1, R_{(c)}^1, R_{(c)}^2$ to write $c_{d+1}(0), c_{d}(1), c_{d-1}(2)$ in terms of the tautological generators. Finally, plugging these into the other three relations
\[
R_{{(b)}}^2,\;R_{{(b)}}^3,\; R_{{(c)}}^3
\]
yields the 3 relations in $\BC[T]^d$. By the non-vanishing of \texttt{det2}, the row-echelon form (with respect to the ordering $\prec$ of monomials) of this  system produces 3 relations $R_1, R_2, R_3$ such that
\begin{equation}\label{R123}
R_i=c_{d-1}(0)c_{4-i}(i-1)+\Big(\textup{sum of monomials}\prec c_{d-1}(0)c_1(2)\Big)\quad \!\! \in \BC[T]^d\,.
\end{equation}
A truncated version (see Remark \ref{rem: truncatedcalculation}) of these relations is implemented in the Mathematica file as \texttt{TruncRelations}.

\section{Truncated relations and proof of the main result}\label{sec: proof}

In this section we will prove the main result of this paper, Theorem \ref{main}. By the symmetry (a) in Section \ref{section0.1}, we may assume without loss of generality that $0<\chi,\, \chi'<d$. Note that when $d<5$ there is nothing to prove, so we assume as well that $d\geq 5$. Suppose that there is an isomorphism of graded $\BC$-algebras $\phi\colon A^\ast(M_{d, \chi})\to A^\ast(M_{d, \chi'})$. We first remark that $\phi$ can be uniquely lifted to the free algebra $\BC[T]$, i.e., there exists a map $\widetilde{\phi}$ fitting into the diagram
\begin{center}
    \begin{tikzcd}
\BC[T]\arrow[dashed,r, "\widetilde{\phi}"]\arrow[d,  twoheadrightarrow] &
\BC[T]\arrow[d,  twoheadrightarrow]\\
A^\ast(M_{d, \chi})\arrow[r, "\phi"]&
A^\ast(M_{d, \chi'})\,.
    \end{tikzcd}
\end{center}
Such unique lift exists since the cohomology is freely generated up to degree $d-2$, where the tautological generators (\ref{taut}) lie. More precisely, the lift is defined by
\[\widetilde{\phi}(c_k(j))\in A^{k+j-1}(M_{d,\chi'})\cong \BC[T]^{k+j-1}\quad\textup{for }k+j\leq d-1\,.\]

By abuse of notation, we write $\phi$ also for the lifted graded ring isomorphism. We will show that unless $\chi \equiv \pm \chi' \mod d$, there is no such lift respecting the three relations in degree $d$, i.e., that sends the 3-dimensional subspace
\[\ker\Big(\BC[T]^d\to A^d(M_{d, \chi})\Big)\subseteq \BC[T]^d \;\;  \textup{ to }\;\; \ker\Big(\BC[T]^d\to A^d(M_{d, \chi'})\Big)\,.\]

\begin{rmk}
A graded ring endomorophism of $\BC[T]$ is defined by an element of 
\[\prod_{k=1}^{d-2}\Hom\big(T^k, \BC[T]^k\big)\,.\]
For $d=5$, the space above has dimension $2\times 2+3\times 6+3\times 13=61$. On the other hand, a 3-dimensional subspace of $\BC[T]^d$ defines a point in the Grassmannian $\mathsf{Gr}\big(3, \BC[T]^d\big)$, which has dimension $3\times (45-3)=126$ for $d=5$. Thus we expect no graded ring endomorphism that takes a general point in $\mathsf{Gr}(3, \mathbb{C}[T]^d)$ to another.
\end{rmk}

Let $R_1, R_2, R_3\in \BC[T]^d$ be the three relations in \eqref{R123}; these span $\ker\Big(\BC[T]^d\to A^d(M_{d, \chi})\Big)$ and are such that the matrix $(R_1\; R_2\; R_3)^\mathsf{T}$ is in row echelon form with respect to the total ordering introduced in Section \ref{sec2.3} on the basis of $\BC[T]^d$. Similarly, we define $R_1', R_2', R_3'$ to be the relations in $\BC[T]^d$ which come from $M_{d,\chi'}$. Since $\phi:\BC[T]^d\xrightarrow{\sim} \BC[T]^d$ preserves the 3-dimensional kernels, it induces an invertible matrix 
$$S=(s_{ij})_{0\leq i,\,j\leq 3}$$ 
such that
\begin{equation}\label{eq: relationscorrespondence}
\phi(R_i)=\sum_{j=1}^3 s_{ij}R_j'\,.
\end{equation}

\subsubsection*{Step 1: truncating relations} We start by truncating the relations $R_i$ by looking only at the terms which are obtained as a product of a generator of degree $d-2$ with a generator of degree $2$. In other words, we consider the projection of $R_i$ to $T^2\otimes T^{d-2}\subseteq \BC[T]^d$ and regard it as a $3\times 3$ matrix $M_i$ by identifying
\[M_i\in T^2\otimes T^{d-2}\cong \Hom((T^{d-2})^\vee, T^2)\cong M_{3\times 3}\,.\]
This is implemented in the Mathematica file as \texttt{TopRelations[d,chi]}. The identification with a $3\times 3$ matrix uses the ordered basis for $T^2, T^{d-2}$ given by the generators from Section~\ref{sec2.3}.

\smallskip

More concretely, 
\[M_i=\Big([c_{3-s}(s)c_{d-1-t}(t)]R_i\Big)_{0\leq s,\,t\leq 2}\]
where $[c_{3-s}(s)c_{d-1-t}(t)]$ means reading off the corresponding coefficient with respect to the monomial basis. The matrices $M_1, M_2, M_3$ can be explicitly calculated as follows:

\begin{align*}
M_1&=\begin{bmatrix}1& 0& 0\\
0& 0& \frac{(d - 2) (d-2\chi) \chi (d - \chi)}{8 d^3}\\
\frac{d - 4}{8 d - 
      16}& \frac{(d-2\chi) \chi (\chi-d)}{2 (d - 2) d^3}& \frac{(-6 \chi - 1) d^5 + (6 \chi^2 + 6 \chi + 
         3) d^4 + 
      18 \chi^2 d^3 - (48 \chi^3 + 48 \chi^2) d^2 + (24 \chi^4 + 
         96 \chi^3) d - 48 \chi^4}{32 (d - 2) d^4}\end{bmatrix},\\
         M_2&=\begin{bmatrix}0& 1& 0\\
1& 0& \frac{(-6 \chi - 1) d^2 + (6 \chi^2 + 12 \chi) d - 12 \chi^2}{8 d^2}\\
0&\frac{24 \chi^2-24 \chi d-d^3+2 d^2}{8 (d-2) d^2}&\frac{\chi \left(d^2+8 d-16\right) (2 \chi-d) (d-\chi)}{8 (d-2) d^3}\end{bmatrix}, \quad \: M_3=\begin{bmatrix}0& 0& 1\\
0& 0& 0\\
\frac{2}{d-2}&0&\frac{12 \chi^2 - 12 \chi d - d^2}{8 (d-2) d}\end{bmatrix}.\\
\end{align*}

\vspace{-16pt}
   
\begin{rmk}\label{rem: truncatedcalculation}
To compute these matrices, and also the extended matrices that will appear later in Section \ref{sec: type1extended}, we follow the strategy described in the previous section. Note that the truncation of the relations $R_{(a)}^n, R_{(b)}^n, R_{(c)}^n$ can be computed since only finitely many terms of \eqref{relation} contribute. For instance, for $\ell=d+1$ the only contributing terms are the ones corresponding to the partitions
\[(d + 1),\, (d, 1),\, (d - 1, 1, 1),\, (d - 1, 2),\, (d - 2, 2, 1),\, (d - 
  2, 3),\, (d - 2, 1, 1, 1)\,.\]
   \end{rmk}
   
Similarly, we define $M_i'$ with $\chi$ replaced by $\chi'$. 
Let $A, B$ be the invertible linear maps
\begin{align*}A&\colon T^{d-2}\hookrightarrow \BC[T]^{d-2}\xlongrightarrow{\phi} \BC[T]^{d-2}\twoheadrightarrow T^{d-2}\,,\\
B&\colon T^{2}\hookrightarrow \BC[T]^{2}\xlongrightarrow{\phi} \BC[T]^{2}\twoheadrightarrow T^{2}\,.
\end{align*}
They are implemented in the Mathematica file as \texttt{AutA} and \texttt{AutB}, respectively.
By considering the standard basis for $T^2, T^{d-2}$ we identify again $A, B$ with $3\times 3$ matrices. More concretely, $A=(a_{st})_{0\leq s,\,t\leq 2}, B=(b_{st})_{0\leq s,\,t\leq 2}$ where 
\begin{align*}
    \phi\big(c_{d-1-s}(s)\big)&=\sum_{t=0}^2 a_{st}c_{d-1-t}(t)+ \textnormal{(lower terms in }\mathbb{C}[T]^{d-2})\,,\\
    \phi\big(c_{3-s}(s)\big)&=\sum_{t=0}^2 b_{st}c_{3-t}(t)+ \textnormal{(lower terms in }\mathbb{C}[T]^{2})\,.
\end{align*}

Then the truncation of \eqref{eq: relationscorrespondence} to $T^2\otimes T^{d-2}$ can be written as an equality between $3\times 3$ matrices
\begin{equation}\label{eq: truncatedcij}
    A^\mathsf{T} M_i B=\sum_{j=1}^3 s_{ij} M_j'\,,\quad i=1,2,3\,.
\end{equation}

\subsubsection*{Step 2: finding $s_{ij}$} We begin with solving $s_{ij}$ for which there are matrices $A, B$ satisfying \eqref{eq: truncatedcij}. By scaling both $A, B$ (hence also $S$) we may assume that $\det(A)=\det(B)=1$. Then taking a linear combination of \eqref{eq: truncatedcij} and applying the determinant it follows that
\begin{equation}\label{eq: det}
    \det\left(\sum_{i=1}^3x_i M_i\right)=\det\left(\sum_{i,j=1}^3 x_i s_{ij} M_j'\right)\,,\end{equation}
 as an equality between two homogeneous cubic polynomials in $x_1, x_2, x_3$. We let $E\subseteq \BP^2$ be the cubic curve defined by
 \[E=\left\{[x_1: x_2:x_3]\colon \det\left(\sum_{i=1}^3x_i M_i\right)\right\}\subseteq \BP^2_{x_1, x_2, x_3}\,.\]
 Similarly define $E'\subseteq \BP^2$. The matrix $S$ defines an automorphism of $\BP^2_{x_1, x_2, x_3}$ sending $E$ to $E'$.

 \begin{lem}
The cubic curves $E, E'$ are elliptic nodal curves with a single node at $[0:0:1]$. Hence $s_{31}=s_{32}=0$. 
 \end{lem}
 \begin{proof}
We can directly compute the equation defining $E$. In the chart \[\BC^2_{x_1, x_2}=\{[x_1: x_2: 1]\}\subseteq \BP^2_{x_1, x_2, x_3}\] the cubic polynomial defining $E$ has the form
\[-\frac{\chi(d-\chi)(d-2\chi)}{4(d-2)d^2}x_1x_2+\textnormal{(cubic terms in }x_1, x_2).\]
It follows that $[0:0:1]$ is indeed a node and the two branches of $E$ at $[0:0:1]$ are tangent to the lines $x_1=0$ and $x_2=0$. A cubic with a node either has exactly one node or has two nodes and is the union of a line and a conic. If the latter were the case, the line would necessarily be $x_1=0$ or $x_2=0$, but this is not possible since $\det(M_2)\neq 0$ and $\det(M_1)\neq 0$. 

Finally, we note that an automorphism of $\BP^2_{x_1, x_2, x_3}$ sending $E$ to $E'$ must preserve the common node $[0:0:1]$, so it follows that $s_{31}=s_{32}=0$.
\end{proof}

We now use equation \eqref{eq: det} to determine the remaining entries of $S$. Given a triple $(u,v,w)$ of non-negative integers such that $u+v+w=3$ we get an equation among the entries of $S$ by comparing the $x_1^ux_2^vx_3^w$ coefficient of both sides of \eqref{eq: det}; this is implemented in the Mathematica file as \texttt{coeff[u,v,w]}{\footnote{In the code, we use \texttt{chi1} and \texttt{chi2} instead of $\chi$ and $\chi'$.}}. The triples $(0,2,1)$ and $(2,0,1)$ give, respectively, 
\[s_{21}s_{22}s_{33}=0\;\;\textup{ and }\;\;s_{12}s_{11}s_{33}=0\,.\]
On the other hand, the equation for $(1,1,1)$ implies that $s_{33}\neq 0$ and $s_{12}s_{21}+s_{11}s_{22}\neq 0$. It follows that solutions must have either:
\begin{enumerate}
    \item[I.] $s_{11}=s_{22}=0$. We call the solutions of this form \textit{Type I} solutions.
    \item[II.] $s_{12}=s_{21}=0$. We call the solutions of this form \textit{Type II} solutions.
\end{enumerate}
From now on we divide the analysis of \eqref{eq: truncatedcij} and \eqref{eq: det} according to the type of the solutions.

\subsection{Type I}
\label{sec: type1}

 Suppose that $S$ is a solution of \textit{Type I} to \eqref{eq: det}, so that we have already the vanishing of the entries $s_{31}=s_{32}=s_{11}=s_{22}=0$. We now obtain the remaining entries. By looking at the $(0,3,0)$ and $(3,0,0)$ equations we obtain $s_{21}, s_{12}$ up to a choice of a cubic root of unity:
\begin{align*}
    s_{21}^3&=-d^3\frac{2 \chi (d-\chi) (d-2\chi) (d^2+8 d -16)}{
 \chi'^2 (d-\chi')^2 (d-2\chi')^2 (d-2)}\,,\\
 s_{12}^3&=-\frac{1}{d^3}\frac{
 \chi^2 (d-\chi)^2 (d-2\chi)^2 (d-2)}{2 \chi' (d-\chi') (d-2\chi') (d^2+8 d -16)}\,.
\end{align*}
Equation $(1,1,1)$ then writes $s_{33}$ in terms of $s_{12}$ and $s_{21}$ and we can conclude that $s_{33}^3=1$. By simultaneously scaling the matrices $B$ and $S$, if necessary, by a cubic root of unity\footnote{Note that we already scaled $B$ so that we could assume that $\det(B)=1$; since scaling by a cubic root of unity does not affect the determinant, we are allowed to do so.} we may assume without loss of generality that $s_{33}=1$. Finally, we obtain $s_{23}$ using $(1,2,0)$ and $s_{13}$ using $(2,1,0)$; note that these are uniquely determined once we choose which cubic root we take for $s_{21}$.


After having obtained $S$, the next step now is to solve equation \eqref{eq: truncatedcij} for the entries of $A$ and $B$. The equation is quadratic in the entries of $A$ and $B$; however, by rewriting it as 
\begin{equation}\label{eq: truncatedinverted}
    A^\mathsf{T} M_i =\left(\sum_{j=1}^3 s_{ij} M_j'\right) B^{-1}\,,\quad i=1,2,3\,,
\end{equation}
 it becomes a linear (homogeneous) system in the entries $a_{st}$ of $A$ and in the entries $\tilde b_{st}$ (written as \texttt{bb[s,t]} in the Mathematica file) of $B^{-1}$. It turns out that the linear system only admits the trivial solution $a_{st}=\tilde b_{st}=0$.
 Indeed, we can express all $a_{st}, \tilde{b}_{st}$ as a multiple of $a_{33}$, and one additional constraint, denoted by \texttt{Diff[d,chi1,chi2][\![1]\!][\![1,2]\!]} in the Mathematica file, imposes that
 \[
 -\frac{2 \chi (d-4) (d-\chi)(d-2\chi) a_{33}}{d ((d-2) d^2+ 24  d{\chi'}-24 {\chi'}^2)}=0,
 \]
 whence $a_{33}=0$. Note that the demoninator is always positive, thus nonzero. This gives a contradiction since $A$ and $B$ are invertible matrices; thus there are no solutions to \eqref{eq: relationscorrespondence} with $S$ of \textit{Type I}. 


\subsection{Type II}
Suppose that $S$ is a solution of \textit{Type II} to \eqref{eq: det}. We already have the vanishing of the entries $s_{31}=s_{32}=s_{12}=s_{21}=0$. The remaining entries are obtained exactly as in the \textit{Type I} case. The equations $(3,0,0)$ and $(0,3,0)$ determine, respectively, $s_{11}$ and $s_{22}$ up to a choice of a cubic root:
\[s_{11}^3=\left(\frac{\chi(d-\chi)(d-2\chi)}{\chi'(d-\chi')(d-2\chi')}\right)^2\,,\quad s_{22}^3=\frac{\chi(d-\chi)(d-2\chi)}{\chi'(d-\chi')(d-2\chi')}\,.\]
Equation $(1,1,1)$ then writes $s_{33}$ in terms of $s_{11}$ and $s_{22}$ and we can conclude that $s_{33}^3=1$. As we did for \textit{Type I}, we may assume without loss of generality that $s_{33}=1$. This fixes $s_{11}=s_{22}^2$. Finally, we obtain $s_{13}$ using $(2,1,0)$ and $s_{23}$ using $(1,2,0)$; note that these are uniquely determined once we choose which cubic root we take for $s_{22}$. 

\smallskip

After obtaining $S$, equation \eqref{eq: truncatedinverted} is again linear in the entries $a_{st}$ of $A$ and in the entries $\tilde b_{st}$ of $B^{-1}$. The linear system can be solved explicitly and it turns out that it has a one-dimensional space of solutions. By writing all the variables in terms of $a_{11}$, the normalization imposed on the determinant $\det(A)=1$ gives
 \[a_{11}^3=\frac{\chi(d-\chi)(d-2\chi)}{\chi'(d-\chi')(d-2\chi')}=s_{22}^3\,.\]
By further scaling $A$ by a cubic root of unity and scaling $B$ by its inverse hence leaving $S$ unchanged, we may assume without loss of generality that $a_{11}=s_{22}$. This completely solves all the matrices $S, A, B$ once we choose which cubic root we take for $s_{11}$. Solutions are implemented in the Mathematica file as \texttt{solS[d,chi1,chi2], solA[d,chi1,chi2], solB[d,chi1,chi2]}, respectively.

\begin{rmk}
When $\chi=\chi'$, the unique solution of \textit{Type II} is given by $S=A=B=\textup{Id}_{3\times 3}$, as expected. When $\chi+\chi'=d$, the solution is given by simple sign matrices
$$A=B=\begin{bmatrix}-1& 0& 0\\
0& 1&0\\
0&0&-1\end{bmatrix},\quad 
S=\begin{bmatrix}1& 0& 0\\
0& -1& 0\\
0&0&1\end{bmatrix}.
$$
In either case, these truncated solutions extend to a solution of \eqref{eq: relationscorrespondence} since the moduli spaces $M_{d, \chi}$ and $M_{d, \chi'}$ are isomorphic. Indeed, the solution matrices are compatible with Proposition \ref{prop2.2}(c) which describes how the tautological generators are mapped under the two types of symmetries.
\end{rmk}


 \subsection{Type II -- extended matrices}
 \label{sec: type1extended}
Recall that we have used only a part of the full equation \eqref{eq: relationscorrespondence} truncated via the projection $\BC[T]^d\twoheadrightarrow T^2\otimes T^{d-2}$. This was sufficient to determine $A, B$, and $S$. 
\subsubsection*{Step 3: bigger truncation}
We proceed by considering a bigger truncation via the projection 
$$\BC[T]^d\twoheadrightarrow \Big(T^2\oplus \Sym^2(T^1)\Big)\otimes T^{d-2}.
$$
We can regard the truncated relations as a block $3\times (3\mid3)$ matrix by identifying 
\[\begin{bmatrix}M_i&N_i\end{bmatrix}
\in \Big(T^2\oplus \Sym^2(T^1)\Big)\otimes T^{d-2}\cong \Hom((T^{d-2})^\vee, T^2\oplus \Sym^2(T^1))\cong M_{3\times 6}\,,\]
where $M_i$ is as before and $N_i$ is implemented in the Mathematica file as \texttt{ExtRelations[d,chi]}. Similarly, we define $\begin{bmatrix}M_i'&N_i'\end{bmatrix}$ with $\chi$ replaced by $\chi'$. As before, $\phi$ induces a linear map
$$\begin{bmatrix}B& U\\
0& V\\\end{bmatrix}\colon T^2\oplus \Sym^2(T^1)=\BC[T]^{2}\xlongrightarrow{\phi} \BC[T]^{2}= T^2\oplus \Sym^2(T^1)\,,
$$
where $U:T^2\rightarrow \Sym^2(T^1)$ and $V:\Sym^2(T^1)\rightarrow \Sym^2(T^1)$. They are implemented in the Mathematica file as \texttt{AutU} and \texttt{AutV}, respectively. Note that the lower left block is zero because $\phi$ is a graded ring isomorphism. 
Then the truncation of \eqref{eq: relationscorrespondence} to $\Big(T^2\oplus \Sym^2(T^1)\Big)\otimes T^{d-2}$ can be written as an equality between $3\times 6$ matrices
\begin{equation*}
    A^\mathsf{T} \begin{bmatrix}M_i&N_i\end{bmatrix} \begin{bmatrix}B& U\\
0& V\\\end{bmatrix}=\sum_{j=1}^3 s_{ij} \begin{bmatrix}M_j'&N_j'\end{bmatrix}\,,\quad i=1,2,3\,.
\end{equation*}
Since $A, B$, and $S$ satisfy \eqref{eq: truncatedcij}, this reduces to 
\begin{equation}\label{eq: reduction}
    A^\mathsf{T}M_iU+A^\mathsf{T}N_iV=\sum_{j=1}^3 s_{ij} N_j'\,,\quad i=1,2,3\,,
\end{equation}
which is a linear system in the entries $u_{st}$ and $v_{st}$ of $U$ and $V$, respectively. 

We show that this linear system, referred to as \texttt{ExtDiff[d,chi1,chi2]} in the Mathematica file, has a solution only if $\chi=\chi'$ or $\chi+\chi'=d$. By using certain parts of the linear system \eqref{eq: reduction}, as explained in the Mathematica file, we can express $u_{11}, u_{21}, u_{31}, v_{11}, v_{21}$ in terms of $v_{31}$. Once this is done, we further use two of the remaining linear system \eqref{eq: reduction}, referred to as \texttt{ExtDiff[d,chi1,chi2][\![1]\!][\![1,1]\!]} and \texttt{ExtDiff[d,chi1,chi2][\![1]\!][\![2,1]\!]}, which take the form 
\begin{align*}
    \texttt{AA[d,chi1,chi2]}v_{31}+\texttt{BB[d,chi1,chi2]}&=0\,,\\
    \texttt{CC[d,chi1,chi2]}v_{31}+\texttt{DD[d,chi1,chi2]}&=0\,,
\end{align*}
following the notations from the code. Existence of the solution $v_{31}$ implies an equation which is purely in terms of $d$, $\chi$ and $\chi'$:
\begin{multline*}
\texttt{Constraint[d,chi1,chi2]}\\ \coloneqq\texttt{AA[d,chi1,chi2]}\cdot \texttt{DD[d,chi1,chi2]}-\texttt{BB[d,chi1,chi2]}\cdot \texttt{CC[d,chi1,chi2]}=0.
\end{multline*}
This expression admits a factorizaton
$$\texttt{Constraint[d,chi1,chi2]}=\texttt{P1[d,chi1]}\cdot \texttt{P2[d,chi1,chi2]}\cdot\texttt{(other terms)}
$$
with $\texttt{(other terms)}$ being clearly nonzero, see the Mathematica file. Note that the roles of $\chi$ and $\chi'$ are symmetric, we conclude that 
\begin{equation}\label{eq: final constraint}
    P_1(d,\chi)\cdot P_2(d,\chi,\chi')=0\quad \textnormal{and}\quad P_1(d,\chi')\cdot P_2(d,\chi',\chi)=0.
\end{equation}
We are left to prove that \eqref{eq: final constraint} implies $\chi=\chi'$ or $\chi+\chi'=d$. Suppose for the contradiction that $\chi\neq \chi'$ and $\chi+\chi'\neq d$. 

\smallskip

We first show that $P_2(d,\chi,\chi')$ and $P_2(d,\chi',\chi)$ are nonzero. From the formula in the file, it is straightforward to check that $P_2(d,\chi,\chi')=0$ if and only if $P_2(d,\chi',\chi)=0$, if and only if 
\begin{equation}
\label{constraint}
d^2 (\chi-\chi') (\chi+\chi'-d)\left(216\chi^4{\chi'}^4-432\chi^3{\chi'}^3(\chi+\chi')d+d^2\cdot f(d,\chi,\chi')\right)=0
\end{equation}
for some explicit integral polynomial $f(d,\chi,\chi')$. Since we assumed that $\chi\neq \chi'$ and $\chi+\chi'\neq d$, the last term must vanish. On the other hand, $\gcd(d,\chi)=\gcd(d,\chi')=1$ implies that $d$ must divide $216$. This further implies that $d^2$ divides $216$ by looking at the linear term in $d$, so we are left with $d=1,2,3,6$. But there are no non-trivial pairs $\chi$ and $\chi'$ for such $d$.

\smallskip
Therefore, we may assume $P_1(d,\chi)=P_1(d,\chi')=0$. We can check that $P_1(d,x)$ is a degree four polynomial in $x$, with the symmetry 
$$P_1(d,x)=P_1(d,d-x).
$$
This implies that $P_1(d,x)$ is a degree four polynomial with four distinct roots 
$$x=\chi,\; \chi',\; d-\chi,\; d-\chi',$$ 
all within the interval $[1,d-1]$. 
On the other hand, explicit computation shows that 
$$P_1(d,0)>0,\quad P_1(d,1)<0$$
as long as $d\geq 5$. This forces $P_1(d,x)$ to have an additional root in the interval $(0,1)$, which contradicts that it is a degree four polynomial. It follows that \eqref{eq: final constraint} implies $\chi=\chi'$ or $\chi+\chi'=d$, hence completing the proof.

\end{document}